\documentclass[a4paper,12pt]{amsart}

\usepackage{amsmath}
\usepackage{amsthm}
\usepackage{amsfonts}
\usepackage{amssymb}
\usepackage[all]{xy}

\newcommand{\ZZ}{\mathbb{Z}}

\usepackage{amssymb,amsmath,amsfonts,amscd,euscript}

\newcommand{\nc}{\newcommand}

\numberwithin{equation}{section}
\newtheorem{thm}{Theorem}[section]
\newtheorem{prop}[thm]{Proposition}
\newtheorem{lem}[thm]{Lemma}
\newtheorem{cor}[thm]{Corollary}
\theoremstyle{remark}
\newtheorem{rem}[thm]{Remark}
\newtheorem{definition}[thm]{Definition}
\newtheorem{example}[thm]{Example}

\nc{\gl}{\mathfrak{gl}}
\nc{\GL}{\mathfrak{GL}}
\nc{\g}{\mathfrak{g}}
\nc{\gh}{\widehat\g}
\nc{\h}{\mathfrak{h}}
\nc{\la}{\lambda}
\nc{\al}{\alpha }
\nc{\ve}{\varepsilon }
\nc{\om}{\omega }

\nc{\ta}{\theta}
\nc{\veps}{\varepsilon}
\nc{\ch}{{\mathop {\rm ch}}}
\nc{\Tr}{{\mathop {\rm Tr}\,}}
\nc{\Id}{{\mathop {\rm Id}}}
\nc{\ad}{{\mathop {\rm ad}}}
\nc{\bra}{\langle}
\nc{\ket}{\rangle}
\nc{\x}{{\bf x}}
\nc{\bs}{{\bf s}}
\nc{\bp}{{\bf p}}
\nc{\bc}{{\bf c}}
\nc{\pa}{\partial}
\nc{\ld}{\ldots}
\nc{\cd}{\cdots}
\nc{\hk}{\hookrightarrow}
\nc{\T}{\otimes}
\newcommand{\bea}{\begin{equation}}
\newcommand{\ena}{\end{equation}}
\nc{\gr}{\mathrm{gr}}
\nc{\ov}{\overline}

\nc{\cO}{\mathcal O}
\nc{\cF}{\mathcal F}
\nc{\cL}{\mathcal L}
\nc{\msl}{\mathfrak{sl}}
\nc{\mgl}{\mathfrak{gl}}
\nc{\U}{\mathrm U}
\nc{\V}{\EuScript V}
\nc{\bH}{\EuScript H}
\nc{\Res}{\mathrm{Res\ }}

\newcommand{\bC}{{\mathbb C}}

\newcommand{\bP}{{\mathbb P}}

\newcommand{\Fl}{\EuScript{F}}

\newcommand{\bV}{{\bf V}}

\newcommand{\bd}{{\bf d}}

\newcommand{\be}{{\bf e}}
\newcommand{\dimv}{{\bf dim}\,}
\newcommand{\Gr}{{\textrm Gr}}
\newcommand{\mn}[3]{{\left(
\begin{matrix}{\displaystyle #1}\\
{\displaystyle #2},\dots ,{\displaystyle #3}\end{matrix}
\right)_{q}
}}

\begin{document}

\title{Schubert Quiver Grassmannians}
\author{Giovanni Cerulli Irelli, Evgeny Feigin, Markus Reineke}
\address{Giovanni Cerulli Irelli:\newline
Sapienza Universit\`a di Roma. Piazzale Aldo Moro 5, 00185, Rome (ITALY)}
\email{cerulli.math@googlemail.com}
\address{Evgeny Feigin:\newline
Department of Mathematics,\newline National Research University Higher School of Economics,\newline
Russia, 117312, Moscow, Vavilova str. 7\newline
{\it and }\newline
Tamm Department of Theoretical Physics,
Lebedev Physics Institute
}
\email{evgfeig@gmail.com}
\address{Markus Reineke:\newline 
Fachbereich C - Mathematik, Bergische Universit\"at Wuppertal, D - 42097 Wuppertal}
\email{mreineke@uni-wuppertal.de}

\begin{abstract}
Quiver Grassmannians are projective varieties para\-metrizing subrepresentations of given dimension in a 
quiver representation. We define a class of quiver Grassmannians generalizing 
those which realize degenerate flag varieties.
%case of subrepresentations of the direct sum of a projective and an injective representation.
%We study this class in detail for the equioriented type $A$ quiver.
We show that
each irreducible component of the quiver Grassmannians in question is isomorphic to a 
Schubert variety. We give an explicit description of the set of irreducible components,
identify all the Schubert varieties arising, and compute the Poincar\'e polynomials of these
quiver Grassmannians.  
\end{abstract}

\maketitle

\section*{Introduction}
Let $R_\bullet=R_1\subseteq R_2\subseteq\cdots\subseteq R_n\subseteq V$ and 
$Q_\bullet=Q_1\subseteq Q_2\subseteq\cdots\subseteq Q_n\subseteq V$ be two flags in a complex vector space 
$V\simeq \mathbf{C}^N$, such that $Q_i\subseteq R_i$ for all $i=1,2,\cdots, n$ and such that  
they are both fixed by a Borel subgroup $B$ of $GL_N$. In this paper we study  
varieties consisting of flags $E_\bullet$ in a partial flag manifold 
$\mathcal{F}l_{(f_1,f_2,\cdots, f_n)}(V)$ of V such that $Q_i\subseteq E_i\subseteq R_i$ 
for all $i=1,2,\cdots, n$. By assumption, these subvarieties of $\mathcal{F}l_{(f_1,f_2,\cdots, f_n)}(V)$ 
are $B$--stable, hence they are unions of finitely many Schubert varieties. 
We show that such varieties arise as suitable quiver Grassmannians of equioriented type 
$A_n$--quivers and we call them \emph{Schubert quiver Grassmannians} (SQG for short). 
This observation allows, on one side, to study the geometry of SQGs via representation 
theory of quivers, and on the other side to study the geometry of such quiver Grassmannians using 
the theory of Schubert varieties. 

Recall that, if $\Gamma$ is a finite quiver with set of vertices $\Gamma_0$, and $M$  is a finite--dimensional complex 
$\Gamma$--representation, then for a dimension vector $\be\in\ZZ^{\Gamma_0}_{\geq0}$, the \emph{quiver Grassmannian} $\Gr_\be(M)$
is the (complex) variety of $\be$--dimensional sub--representations of $M$.  In \cite{CFR}, the case when 
$\Gamma$ is Dynkin (i.e. an orientation of a simply--laced Dynkin diagram) and $M=P\oplus I$ -- the sum of a 
projective and an injective representation of $\Gamma$ -- was studied. More precisely,
it was shown that the varieties $\Gr_{\dimv P}(P\oplus I)$ enjoy many nice properties.
In particular, they are irreducible varieties endowed with the action of a large
abelian unipotent group, acting with an open dense orbit. An example of such varieties
is provided by the type $A$ degenerate flag varieties \cite{F1}, \cite{F2}. In this case, 
$\Gamma$ is the equioriented type $A$ quiver, $P$ is the direct sum of all indecomposable projective representations 
and $I$ is the direct sum of all indecomposable injective representations.
It was shown in \cite{CL}, \cite{CLL} that the above mentioned abelian unipotent group can be identified with a certain subgroup of
a Borel of the group $SL_N$ (note that $N$ here is much larger than the number of vertices
of the initial quiver). Moreover, the quiver Grassmannian itself can be identified with 
a Schubert variety for $SL_N$. The main reason why the Borel subgroup appears is that the indecomposable
direct summands of $M=P\oplus I$ form a chain (that is, a totally ordered subset) with respect to the partial ordering induced by the  Auslander-Reiten quiver of $\Gamma$.

This observation motivates the following more general definition:
Let $\Gamma$ be a Dynkin quiver  and let $M_1,\dots, M_r$
be a set of indecomposable representations forming a chain with respect to the partial ordering induced by the Auslander-Reiten quiver.
%such that they all belong to a same oriented and connected subpath 
%of the Auslander-Reiten quiver of $\Gamma$. In particular, this condition is fulfilled if any $M_i$ is 
%either projective or injective.
We call a direct sum $M=\bigoplus_{i=1}^r M_i^{\oplus a_i}$ \emph{catenoid}.

From now on let $\Gamma$ be a quiver of equioriented type $A_n$. In this case, there is a natural way to embed every quiver Grassmannian $Gr_\be(M)$ into a flag manifold: consider a minimal projective resolution 
$$
\xymatrix{
0\ar[r]& Q\ar^\iota[r]& R\ar[r]&M\ar[r]& 0
}
$$
of a representation $M$. Then $Gr_\mathbf{e}(M)$ is isomorphic (as a scheme) to the variety $Gr_\mathbf{f}(\xymatrix@1@C=15pt{Q\ar|-{\iota}[r]&R})$ consisting of sub--representations of $R$ of dimension vector $\mathbf{f}:=\be+\mathbf{dim}\, Q$ containing $\iota(Q)$ (see Proposition~\ref{Prop:IsoQuiveFlagVar}). Since all the maps defining $Q$ and $R$ are injective, they induce flags inside the vector space $R_n$, and hence the quiver Grassmannian $Gr_\mathbf{e}(M)$ embeds into a flag manifold of $R_n$. We refer to this construction as the \emph{natural embedding of a quiver Grassmannian inside a flag manifold}. Notice that $R_n\simeq\mathbf{C}^N$ where $N=a_1+\cdots+a_r$ is the number of indecomposable direct summands of M. 
Our first result is the following.  
\begin{thm}
The natural embedding of $Gr_\be(M)$ inside a flag manifold is stable under the action of a Borel $B_N\subset GL_N$ if and only if M is catenoid. 
In this case, $Gr_\be(M)$ is a SQG and every SQG arises in this way. The irreducible components of SQGs are Schubert varieties and can be explicitly described. The Schubert cells -- the $B_N$-orbits -- provide a cellular decomposition of 
$\Gr_\be (M)$ and they are stable under the action of $\mathrm{Aut}(M)$; in particular the points of a given cell are all isomorphic as representations of $\Gamma$. 
\end{thm}
%It should be noted that the partial flag variety appearing here only depends on 
%the representation $M$, and not on the dimension vector $\be$.

It is a natural question to ask when a SQG is a Schubert variety, or, in other words, when it is irreducible. 
In general, a SQG can be factored into a product of several SQGs associated with representations of smaller quivers 
of equioriented type $A$. If this is not the case we say that the SQG is \emph{simple} 
(see Definition~\ref{Def:Simple}). It is hence enough to answer this question for a simple SQG. 
\begin{thm}
Let $M$ be catenoid. A simple SQG $Gr_\mathbf{e}(M)$  is a Schubert variety if and only if there exists a projective representation $P$ and an injective representation $I$ such that 
$M$ fits into an exact sequence $\xymatrix@1@C=8pt{0\ar[r]&P\ar[r]&M\ar[r]&I\ar[r]&0}$ and $\mathbf{e}=\mathbf{dim}\,P$. In this case, every point of the open $B_N$-orbit corresponds to a subrepresentation of $M$ isomorphic to $P$. 
\end{thm}
Given a catenoid representation $M$ and a dimension vector $\be$, we give a 
combinatorial description of the
set of irreducible components of $\Gr_\be(M)$. For each irreducible component, we identify the
partial flag variety containing the corresponding Schubert variety, and the corresponding Weyl group element. We show that the natural action of $\mathrm{Aut}(M)$ on  $Gr_\mathbf{e}(M)$ becomes, after  the natural embedding into a flag manifold,  the action of the parabolic subgroup of $GL_N$ consisting of all automorphisms of $R$ which fix $\iota(Q)\subseteq R$. 
We also compute the Poincar\'e polynomials of SQGs.  

The paper is organized as follows. In Section \ref{setup} we fix the notation and introduce 
the objects of study. In Section \ref{Schubert} we provide a link to the theory of Schubert varieties.
In Section~\ref{ic} we describe the irreducible components of SQGs.
In Section \ref{Euler} we compute their Poincar\'e polynomials.

\section{The setup}\label{setup}
\subsection{General definitions}
Let $\Gamma=\Gamma_n:=\xymatrix{1\ar[r]&2\ar[r]&\cdots\ar[r]&n}$ be an equioriented quiver of type $A_n$, for some fixed integer $n\geq1$.
The indecomposable representations for this quiver
are parametrized by (integer) subintervals  $[i,j]$ of the interval $[1,n]$ ($1\leq i\leq j\leq n$). We denote the corresponding indecomposable representation by $M[i,j]$, 
and we say that the \emph{support} of $M[i,j]$ is the interval $[i,j]$. 
Any representation $M$ can be written in an essentially unique way as  
$M=\bigoplus_{1\leq i\leq j\leq n} M[i,j]^{\oplus m_{ij}}$ for some non--negative integers $m_{ij}$. 
We denote the dimension vector
$\mathbf{d}=(d_1,\cdots, d_n)$ of $M$ by $\dimv\, M$. We note that, for every vertex $i$, 
the projective cover of the simple representation
$S_i=M[i,i]$ is $P_i:=M[i,n]$, and its injective envelope is $I_i:=M[1,i]$.

The Auslander--Reiten quiver of $\Gamma$ is described as follows: its vertices are parametrized by pairs $(i,j)$ for $1\le i\le j\le n$, and the arrows
are $(i,j)\to (i-1,j)$ and $(i,j)\to (i,j-1)$ (whenever the targets are well defined).

We are interested in the special class of $\Gamma$--representations defined as follows. 
\begin{definition}
We say that a $\Gamma$--representation $M$ is \emph{catenoid} if all the distinct indecomposable direct summands of $M$ belong to an oriented  
connected path of the Auslander--Reiten quiver of $\Gamma$. 
\end{definition}

Recall that a subset $\bp\subset \{(i,j): 1\le i\le j\le n\}$ is a \emph{Dyck path} on $[1,n]$ if
we can order the elements of $\bp$ as $(p_1,\dots,p_r)$ in such a way that $p_1=(1,1)$, $p_r=(n,n)$, and
if $p_m=(i_m,j_m)$, then $p_{m+1}=(i_m,j_m+1)$ or $p_{m+1}=(i_m+1,j_m)$. We say that a representation $M$ is supported on $\bp$ if $(i,j)\in\bp$ for each indecomposable direct summand $M[i,j]$
of $M$. The oriented connected paths of the Auslander-Reiten quiver of $\Gamma$ are precisely the Dyck paths (on $[1,n]$).
It follows that $M$ is catenoid if and only if it is supported on a single Dyck path.  

Another way to view this definition is the following. Consider the set $\mathcal{P}([1,n])$ of all connected 
sub-intervals of the interval $[1,n]$ (those are precisely the vertices of the Auslander-Reiten quiver of $\Gamma$). We endow $\mathcal{P}([1,n])$ with the following partial ordering: 
\begin{equation}\label{Eq:DefPartOrder1}
\xymatrix{
[i,j]\leq [k,\ell]\ar@{<=>}^(.45){def}[r]&i\leq k\textrm{ and }j\leq\ell
}
\end{equation}
For example, for $n=4$, we have:
$$
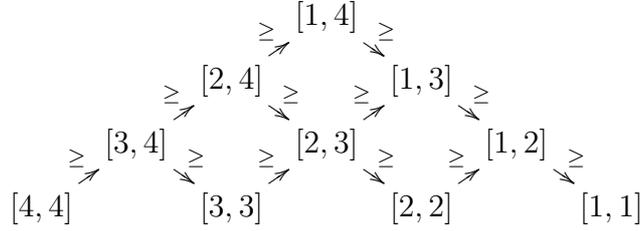
\xymatrix@C=5pt@R=5pt{
&&&[1,4]\ar[dr]^\geq&&&\\
&&[2,4]\ar[dr]^\geq\ar[ur]^\geq&&[1,3]\ar[dr]^\geq&&\\
&[3,4]\ar[dr]^\geq\ar[ur]^\geq&&[2,3]\ar[dr]^\geq\ar[ur]^\geq&&[1,2]\ar[dr]^\geq&\\
[4,4]\ar[ur]^\geq&&[3,3]\ar[ur]^\geq&&[2,2]\ar[ur]^\geq&&[1,1]
}
$$
We have an induced partial order on the set of (isoclasses of) indecomposable $\Gamma$--representations: 
\begin{equation}\label{Eq:DefPartOrder}
\xymatrix{
M[i,j]\leq M[k,\ell]\ar@{<=>}^(.55){def}[r]&[i,j]\leq [k,\ell].
}
\end{equation}

\begin{lem}
$M[i,j]\leq M[k,\ell]$ if and only if there exists an oriented (connected) path from 
$M[k,\ell]$ to $M[i,j]$ in the Auslander--Reiten quiver of $\Gamma$. 
In particular, a representation $M$ is catenoid if and only if the set of its distinct direct summands form a totally ordered 
set with respect to the partial order \eqref{Eq:DefPartOrder}. 
\end{lem}
\begin{proof}
Let us denote the arrows $(x,y)\rightarrow (x-1,y)$ of the Auslander--Reiten quiver by $\alpha_+$, and the arrows $(x,y)\rightarrow (x,y-1)$ by $\alpha_-$. Let $(k,\ell)$ and $(i,j)$ be two vertices. Then there is a path from $(k,\ell)$ to $(i,j)$ if and only if $i=k-s$ and $j=\ell-t$ for some $0\leq s\leq k$ and $0\leq t\leq \ell$. Indeed, every connected path starting at $(k,\ell)$  is obtained from the path $p=\alpha_+^s\circ \alpha_-^t$ by applying repeatedly the local replacement $\xymatrix@1@C=10pt{\alpha_-\circ \alpha_+\,\ar@{~>}[r]&\,\alpha_+\circ \alpha_-}$.  This shows that there is a path from $(k,\ell)$ to $(i,j)$ if and only if $(i,j)\leq (k,\ell)$ as desired. The rest follows. 
\end{proof}

\begin{example}
Examples of catenoid representations are $P\oplus I$, where $P$ is projective and $I$ is injective.
In \cite{CFR,CFR2}, the quiver Grassmannians $\Gr_{\dimv P}(P\oplus I)$ were studied (see Example~\ref{Ex:DegFlag}). Another class of examples is provided by direct sums of one- and two-dimensional indecomposables, or, in other words, representations of $\Gamma$ forming a complex (see Example~\ref{Ex:Complexes}).
\end{example}

\section{Quiver Grassmannians associated with catenoid representations}\label{Schubert}
We retain notation of the previous section. So, $\Gamma$ is an equioriented quiver of type $A_n$. Given a $\Gamma$--representation $M$ and a dimension vector $\be\in\ZZ_{\geq0}^n$ we denote by $\Gr_\be(M)$ the projective varieties consisting of all subrepresentations of $M$ of dimension vector $\be$ (see \cite[Section~2]{CFR} for details).

There is a natural way to realize $\textrm{Gr}_\mathbf{e}(M)$ as a closed subvariety of 
a partial flag manifold, generalizing the idea of \cite{CL}. Let us describe this realization. 
A minimal projective resolution of $M$ is given by 
\begin{equation}\label{Eq:ProjReso}
\xymatrix@C=8pt{0\ar[r]&Q\ar^\iota[r]&R\ar^\pi[r]&M\ar[r]&0}
\end{equation}
where $R=\bigoplus_{1\leq i\leq n}P_i^{[M,S_i]}$ and $Q=\bigoplus_{1\leq i\leq n}P_i^{[M,S_i]^1}$, and we use the standard notation 
$$
\begin{array}{ccc}
[M,N]:=\textrm{dim Hom}_\Gamma(M,N)&\textrm{and}&[M,N]^1:=\textrm{dim Ext}^1_\Gamma(M,N).
\end{array}
$$
We note that the arrows of $\Gamma$ act as injective linear maps on every projective 
$\mathbf{C}\Gamma$--module (recall that the category of $\Gamma$--representations is equivalent to the category 
of modules over the path algebra of $\Gamma$, which is the $\mathbf{C}$--vector space generated by paths of $\Gamma$ with obvious multiplicative structure), and hence 
$R$ is a partial flag  
$R=R_\bullet:=R_1\subseteq R_2\subseteq\cdots \subseteq R_{n-1}\subseteq R_n$. 
Here the $R_i$'s are finite dimensional 
vector spaces of dimension ${\rm dim}\, R_i=:r_i=\sum_{k\leq i}[M,S_k]$ (for all $i=1,\cdots,n$). 
In particular the dimension $N$ of the vector space $R_n$ is precisely the number of indecomposable direct summands of $M$.
Similarly, the representation $Q$ is a partial flag 
$Q=Q_\bullet:=Q_1\subseteq Q_2\subseteq\cdots \subseteq Q_{n-1}\subseteq Q_n$, where $Q_i$ is a  vector subspace of $R_i$ 
of dimension 
$
{\rm dim}\, Q_i=:q_i=\sum_{k\leq i}[M,S_k]^1
$ 
(for all $i=1,\cdots,n$).
Note that the vector space $Q_1$ is zero: indeed, $[M,S_1]^1=0$ since $S_1$ is injective. 

For a vector space $V$ of dimension $m$, and an increasing sequence $\mathbf{f}=(f_1,f_2,\cdots, f_n)$ of non--negative integers 
$0\leq f_1\leq f_2\leq\cdots\leq f_n$ we denote by $\textrm{Fl}_\mathbf{f}(V)$ the variety of partial flags 
$\mathcal{U}_\bullet:=(\mathcal{U}_1\subseteq\mathcal{U}_2\subseteq\cdots\subseteq\mathcal{U}_n\subseteq V)$ in V with $\textrm{dim } \mathcal{U}_i=f_i$.

We consider the dimension vector $\mathbf{f}:=\mathbf{e}+\mathbf{dim}\,Q$, and  the  variety 
$Gr_\mathbf{f}(\xymatrix@1@C=15pt{Q\ar|-{\iota}[r]&R})\subset \mathcal{F}l_\mathbf{f}(R_n)$ of partial flags $\mathcal{U}_\bullet$ in $R_n$
such that
\begin{equation}
\iota_j(Q_j)\subseteq\mathcal{U}_j\subseteq R_j\qquad(j=1,\cdots, n).
\end{equation}
The variety $Gr_\mathbf{f}(\xymatrix@1@C=15pt{Q\ar|-{\iota}[r]&R})$ is a closed subvariety of 
$\textrm{Fl}_\mathbf{f}(R_n)$, and the following result provides the required closed embedding 
$\Gr_\mathbf{e}(M)\subset \textrm{Fl}_\mathbf{f}(R_n)$. 
\begin{prop}\label{Prop:IsoQuiveFlagVar}
The map $\varphi:Gr_\mathbf{e}(M)\rightarrow \mathcal{F}l_{\mathbf{e}+\mathbf{dim}\,Q}(R_n)$  defined by
$$
\varphi(N)=(\pi^{-1}(N))_1\subseteq (\pi^{-1}(N))_2\subseteq\cdots\subseteq(\pi^{-1}(N))_n
$$
is a closed embedding of schemes. The image is the variety $Gr_\mathbf{f}(\xymatrix@1@C=15pt{Q\ar|-{\iota}[r]&R})$.
\end{prop}
\begin{proof}
Since $\pi$ is surjective, the map $N\mapsto \pi^{-1}(N)$ is an isomorphism of $\Gr_\mathbf{e}(M)$ with the variety 
 of all sub--representations of $R$ which contain $\textrm{Ker }\pi=\iota(Q)$,  
and whose dimension vector is $\mathbf{f}=\mathbf{e}+\dimv\, \textrm{Ker}\pi$. It is straightforward to verify that it is a closed embedding of schemes (see e.g. \cite[Section~3.2]{Haupt}). 
\end{proof}

In the special case when the subvariety $Gr_\mathbf{f}(\xymatrix@1@C=15pt{Q\ar|-{\iota}[r]&R})$ of $\mathcal{F}l_\mathbf{f}(R_n)$ is closed under the action of a Borel subgroup of $GL(R_n)$, we call it a \emph{Schubert quiver Grassmannian} (SQG for short). The following result provides a characterization of SQGs. 
%\begin{definition}
%We say that the quiver Grassmannian $\Gr_\mathbf{e}(M)$ is \emph{Schubert} if the corresponding variety  
%$Gr_\mathbf{e}(Q\subseteq R)\subset\textrm{Fl}_\mathbf{e}(R_n)$ is stable under the action of a Borel subgroup B of $\textrm{GL}(R_n)$. 
%\end{definition}

\begin{thm}\label{Thm:SchubQuivGrass}
The subvariety $Gr_\mathbf{f}(\xymatrix@1@C=15pt{Q\ar|-{\iota}[r]&R})\subseteq \mathcal{F}l_\mathbf{f}(R_n)$ is closed under the action 
of a Borel subgroup of $GL(R_n)$ if and only if $M$ is catenoid. 
\end{thm}
\begin{proof}
We show that $M$ is catenoid if and only if 
there exists a complete flag $F_\bullet$ which is a refinement of both $R_\bullet$ and $Q_\bullet$.
In other words, $M$ is catenoid if and only if there exists an 
ordered basis $\{v_1,\cdots, v_{N}\}$ of the vector space $R_n$ such that $R_i=\textrm{span}\{v_1,\cdots, v_{r_i}\}$ 
and $Q_i=\textrm{span}\{v_1,\cdots, v_{q_i}\}$ for every $i=1,\cdots, n$. This is equivalent to the fact  
that both $R_\bullet$ and $Q_\bullet$ are fixed by the  Borel subgroup of upper triangular matrices, and hence the claim holds true.

If there exists such a basis, then the $\mathbf{C}\Gamma$--module structure 
on $R$ is given naturally as follows: the underlying vector space is $\oplus_{i=1}^nR_i$ generated by the induced 
basis 
\[
\mathcal{B}=\{v_1^{(1)},\cdots, v_{r_1}^{(1)}\}\cup\{v_1^{(2)},\cdots, v_{r_2}^{(2)}\}\cup\cdots\cup\{v_1^{(n)},\cdots, v_{N}^{(n)},\}
\] 
and the action of the path $(i\rightarrow i+1\rightarrow \cdots\rightarrow j)\in\mathbf{C}\Gamma$ is given by 
the canonical embedding sending $v_k^{(i)}\mapsto v_{k}^{(j)}$. For every $k=1,\cdots, N$, let $R(k)$ be 
the submodule generated by all the basis vectors $v_k^{(i)}$ for $i=i(k),\cdots, n$, where $i(k)$ is the 
minimal vertex of $\Gamma$ such that $k\leq r_{i(k)}$ so that $v_k^{(i(k))}$ is defined. Clearly, $R(k)$ is a 
projective direct summand of $R$ isomorphic to $P_{i(k)}$. Similarly, the module $Q$ admits, by hypothesis, 
a sub-basis $\mathcal{B'}\subset\mathcal{B}$, and we define the submodule $Q(k)$ similarly to $R(k)$. 
Then $Q(k)$ is isomorphic to the projective $P_{j(k)}$, where $j(k)$ is the minimal vertex  such that 
$k\leq q_{j(k)}$ so that $v_k^{(j(k))}$ belongs to $Q$. By construction, $Q(k)$ embeds into $R(k)$,  
and the quotient is $M[i(k),j(k)-1]$. Hence all the indecomposable direct summands 
of $M$ are of the form $M[i(k),j(k)-1]$, for $k=1,\cdots, N$. By construction, we have
$$
\xymatrix{
i(k)\leq i(s)\ar@{<=>}[r]& k\leq s\ar@{<=>}[r]&j(k)\leq j(s)},
$$
It follows that, given two distinct direct summands $M[i(k),j(k)-1]$ and $M[i(s),j(s)-1]$ of $M$ such that $i(k)<i(s)$, we have $k<s$, and hence $j(k)<j(s)$. This proves that all the direct summands of $M$ 
are comparable with respect to the ordering \eqref{Eq:DefPartOrder}, and hence $M$ is catenoid by definition. 

Now suppose that, conversely, $M=\bigoplus_{l=1}^r M(l)^{\oplus a_l}$ is catenoid with indecomposable direct summands $M(1), M(2),\cdots, M(r)$ ordered so that
$M(l)<M(l+1)$ in the partial order \eqref{Eq:DefPartOrder}.  Their tops 
are the simples $S(i_1)$, $S(i_2),\cdots, S(i_r)$ for some indices 
$1\leq i_1\leq i_2\leq\cdots\leq i_r\leq n$, respectively, and their socles are $S(j_1)$, $S(j_2),\cdots, S(j_r)$  for 
some vertices $1\leq j_1\leq j_2\leq\cdots\leq j_r\leq n$, respectively. A minimal projective resolution of 
$M$ is obtained as the direct sum of the minimal projective resolutions of the $M(l)$'s, namely 
$$
\xymatrix{
0\ar[r]&Q(l)\ar^{\iota_l}[r]&R(l)\ar^{\pi_l}[r]&M(l)\ar[r]&0
},
$$
where $R(l)=P_{i(l)}$, $Q_l=P_{j(l)+1}$ and $\iota_l$ and $\pi_l$ are the canonical homomorphisms. 
For every vertex $i$, the vector space $R_i$ is the direct sum of the one--dimensional vector spaces 
$R(1)_i, R(2)_i, \cdots, R(r_i)_i$, and hence it admits a basis $\mathcal{B}_i$ given by 
$v_1^{(i)}, v_2^{(i)}, \cdots, v_{r_i}^{(i)}$, where $v_l^{(i)}$ is the generator of $R(l)_i$. 
By construction, the vector subspace of $R_i$ generated by  $v_1^{(i)}, v_2^{(i)}, \cdots, v_{q_i}^{(i)}$ is $Q_i$. 
\end{proof}

%We close with an example.
%\begin{example}
%Let $M=M[1,4]\oplus M[2,3]$ as representation of $A_4$-equioriented. Then M is not a catenoid and its projective resolution looks like: 
%$$
%\xymatrix@R=3pt@C=8pt{
%&\bullet\ar@{-}[r]&\bullet\ar@{-}[r]&\ast\\
%\bullet\ar@{-}[r]&\bullet\ar@{-}[r]&\bullet\ar@{-}[r]&\bullet
%}
%$$
%(In this picture $\bullet$ denotes the vectors of $R$ while the $\ast$ denote the vector of Q). 
%Note that the set of vertices $\ast$ is not  closed under movements towards the bottom. On the 
%other hand $M=M[1,2]\oplus M[2,4]$ is a catenoid and its projective resolution looks like:
%$$
%\xymatrix@R=3pt@C=8pt{
%&\bullet\ar@{-}[r]&\bullet\ar@{-}[r]&\ast\\
%\bullet\ar@{-}[r]&\bullet\ar@{-}[r]&\ast\ar@{-}[r]&\ast
%}
%$$
%Now the set of vertices $\ast$ is closed under "moving down". 
%\end{example}

\subsection{Group action on SQGs}
Let $M$ be a $\Gamma$--representation with minimal projective resolution \eqref{Eq:ProjReso}. 
Let $\mathbf{e}$ be a dimension vector for $\Gamma$ and consider the quiver Grassmannian $Gr_\mathbf{e}(M)$. 
The group $\mathrm{Aut}(M)$ (consisting of $\Gamma$--automorphisms of M) naturally acts on $Gr_\mathbf{e}(M)$. 
On the other hand, the closed subvariety $Gr_\mathbf{f}(\xymatrix@1@C=15pt{Q\ar|-{\iota}[r]&R})$ of 
$\mathcal{F}l_\mathbf{f}(R_n)$ (where $\mathbf{f}:=\mathbf{e}+\mathbf{dim}\,Q$) is acted upon by the following group: 
$$
\mathcal{P}_M:=\{f\in Aut(R)|\, f\circ\iota=\iota\}
$$
which consists of those automorphisms of $R$ which fix $\iota(Q)\subseteq R$. Since $Gr_\mathbf{e}(M)$ and 
$Gr_\mathbf{f}(\xymatrix@1@C=15pt{Q\ar|-{\iota}[r]&R})$ are isomorphic under the morphism $\varphi$ of 
Proposition~\ref{Prop:IsoQuiveFlagVar}, it is natural to ask how those two actions are related. 
The next result says that they  coincide. 
\begin{prop}\label{Prop:Action}
There is an exact sequence of groups
$$
\xymatrix{
1\ar[r]&K\ar[r]&\mathcal{P}_M\ar^(.4)\Psi[r]&Aut(M)\ar[r]&1
}
$$
where $K=1_R+\mathrm{Hom}_\Gamma(R,Q)$ and it acts trivially on 
$Gr_{\mathbf{e}+\mathbf{dim}\,Q}(\xymatrix@1@C=15pt{Q\ar|-{\iota}[r]&R})$. In particular the 
action of $\mathcal{P}_M$ on $Gr_{\mathbf{e}+\mathbf{dim}\,Q}(\xymatrix@1@C=15pt{Q\ar|-{\iota}[r]&R})$ 
coincides with the action of $\mathrm{Aut}(M)$ on $Gr_\mathbf{e}(M)$ under the isomorphism $\varphi$. 
\end{prop}
\begin{proof}
We define the map $\Psi$ as follows: let $f\in\mathcal{P}_M$. For $m\in M$ we define $\Psi(f)(m):=\pi\circ f(\pi^{-1}(m))$. 
First of all, we notice that $\Psi$ is well--defined: indeed if $r,r'\in\pi^{-1}(m)$ we have $r-r'\in \iota(Q)$ and hence 
$\pi(f(r-r'))=\pi\circ f\circ\iota(r-r')=\pi\circ\iota (r-r')=0$. Moreover $\Psi$ is easily seen to be  
a homomorphism of groups. Let us show that $\Psi$ is onto: let $g\in Aut(M)$, then, since $\pi:R\rightarrow M$ 
is a minimal projective resolution of M, there exists $f\in Aut(R)$ such that $g\circ\pi=\pi\circ f$ 
and thus $g=\Psi(f)$. Now, suppose that $\Psi(f)=1_M$; this means that $\pi\circ f-\pi=0$ or in 
other words: $\pi(f-1_R)=0$ which means that $f-1_R\in Hom_\Gamma(R,Q)$. 

Let $T\in Gr_{\mathbf{e}+\mathbf{dim}\,Q}(\xymatrix@1@C=15pt{Q\ar|-{\iota}[r]&R})$ 
(so that $\pi(T)\in Gr_\mathbf{e}(M)$), and let $f\in \mathcal{P}$. We need to check that
$$
\pi\circ f (T)=\Psi(f)(\pi(T)).
$$
This follows immediately from the definition of $\Psi$ and concludes the proof. 
\end{proof}

We now turn back our attention to catenoids. Let $M=\bigoplus_{l=1}^r M(l)^{\oplus a_l}$ be catenoid. 
Assume that each $M(l)$ is indecomposable and $M(l)<M(l+1)$ in the partial order \eqref{Eq:DefPartOrder}. 
Let us denote by $N=a_1+\cdots+a_r$ the number of indecomposable direct summands of M and we number 
them as $M_1\leq M_2\leq \cdots\leq M_N$ (each $M_i$ is isomorphic to some $M(l)$). We fix an element 
$\phi_{i,j}\in {\rm Hom}_\Gamma(M_j, M_i)$, $1\le i, j\le N$  in each homomorphism space (which is at most one--dimensional);
we assume that $\phi_{i,j}\ne 0$ if ${\rm Hom}_\Gamma(M_j, M_i)$ is non zero. 
It is worth noting that $\phi_{ij}$ can be zero even if $i<j$.  
The elements $\phi_{i,j}$, considered as elements of ${\rm Hom}_\Gamma(M, M)$, span the vector space
${\rm Hom}_\Gamma(M, M)$. In particular, every element $\sum_{i,j} a_{ij}\phi_{ij}$ of $\mathrm{Aut}(M)$ 
comes from an $N\times N$ matrix $\sum_{i,j}a_{ij}E_{ij}$ (where $E_{ij}$ is the elementary matrix with 1 in place 
$(i,j)$ and zero elsewhere); so the whole automorphism group of $M$ can be seen as the image of the
standard parabolic subgroup $P\subset GL_N$ with blocks of sizes $a_1,\dots,a_r$.
We thus obtain the homomorphism $P\to Aut(M)$. In Proposition \ref{Prop:ActionCatenoid} we identify
$P$ with $\mathcal{P}_M$ of Proposition \ref{Prop:Action}. We will need the following lemma.
%$$
%\left(
%\begin{array}{cccc}
%\mathrm{Aut}(M(1)^{a_1})&\Hom(M(2)^{a_2},M(1)^{a_1})&\cdots&\Hom(M(r)^{a_r},M(1)^{a_1})\\
%0&\mathrm{Aut}(M(2)^{a_2})&\cdots&\Hom(M(r)^{a_r},M(2)^{a_2})\\
%\vdots&\vdots&\ddots&\vdots\\
%0&0&0&\mathrm{Aut}(M(r)^{a_r})
%\end{array}
%\right)
%$$
\begin{lem}\label{LU}
The group $\mathrm{Aut}(M)$ is generated by its Levi part $L(M)=\prod_{l=1}^r {\rm Aut} (M(l)^{a_l})$ and
its unipotent part $U(M)$, generated by the one-parameter subgroups $\exp(z\phi_{i,j}, z\in \bC)$
such that $M_i$ and $M_j$ are not isomorphic. 
\end{lem} 
\begin{proof}
An automorphism of $M$ is completely determined by restrictions to the indecomposable summands.
This implies the lemma.
\end{proof}

We now provide another characterization of catenoids. Let $M=\bigoplus_{l=1}^r M(l)^{\oplus a_l}$ be a 
$\Gamma$--representation (not necessarily catenoid) with minimal projective resolution \eqref{Eq:ProjReso} 
and indecomposable direct summands $M_1,M_2,\cdots, M_N$. The automorphism group of $R$ can be naturally 
embedded into $GL_N$ as follows: We denote the projective cover of each $M_i$ by $P^i$ so that 
$R=P^1\oplus P^2\oplus\cdots\oplus P^N$. Without loss of generality, we can assume that 
$P^1\leq P^2\leq \cdots \leq P^N$ in the partial order \eqref{Eq:DefPartOrder}. We denote by $\rho_{ij}$ 
a generator of the vector space  $\mathrm{Hom}_\Gamma(P^j,P^i)$ and by $\Omega=\{(i,j)|\,\rho_{ij}\neq0\}$.  
Then an element $f\in \mathrm{Aut}(R)$ has the form $\sum_{(i,j)\in\Omega} b_{ij}\rho_{ij}$ and we send 
this element to the matrix $\sum_{i,j} b_{ij}E_{ij}\in GL_N$. Notice that $\mathrm{Hom}_\Gamma(P^j,P^i)$ 
is one-dimensional for every $1\leq i\leq j\leq N$, and hence the set of all matrices
$\sum_{(i,j)\in\Omega} b_{ij}E_{ij}$  form a parabolic 
subgroup of $GL_N$ that we denote by $\mathcal{R}$ ($\mathcal{R}$ contains the Borel subgroup of upper triangular matrices). 
Under the isomorphism $\mathrm{Aut}(R)\rightarrow \mathcal{R}$, we denote by $P$ the image of  the subgroup 
$\mathcal{P}_M$ defined above.  
\begin{prop}\label{Prop:ActionCatenoid}
Let M be a $\Gamma$--representation. Then M is catenoid if and only if $P$ is a parabolic subgroup of $GL_N$. 
In this case $P$ is the standard parabolic subgroup of $GL_N$ associated with the partition $N=a_1+\cdots+a_r$ 
with Levi decomposition $P=LU$. The restriction of 
$\Psi$ to $L$ is an isomorphism to $L(M)$ and $\Psi(U)=U(M)$.  
\end{prop}
\begin{proof}
We know from the proof of Theorem~\ref{Thm:SchubQuivGrass} that M is catenoid if and only if the two flags 
$R_\bullet$ and $Q_\bullet$ are both stable under the action of a same Borel subgroup $B_N\subseteq GL_N$. 
In particular, $B_N\subseteq \mathcal{R}$ stabilizes $Q_\bullet$ and hence $B_N\subseteq P$. 

In the following, we use the notation $\beta_k:=a_1+\cdots+a_k$. Let $P'$ be the standard parabolic subgroup of 
$GL_N$ associated with the partition $N=a_1+\cdots+a_r$. We define the map $\xi: P'\rightarrow\mathrm{Aut}(R)$ 
which sends an element
$g=\sum_{i,j=1}^N g_{i,j}E_{i,j}\in P'$ to the element $\xi(g):=\sum_{i,j=1}^N g_{i,j}\rho_{i,j}\in \mathrm{Aut}(R)$. 
First of all, $\xi$ is well--defined: indeed every pair $(i,j)$ of the form $\beta_k\leq j<i<\beta_{k+1}$ 
(for some $k=1,2,\cdots, r-1$) belongs to $\Omega$ since $P^i=P^j$ is the projective cover of $M(k)$ and hence 
$\rho_{ij}\neq0$. Clearly, $\xi$ is an injective homomorphism of groups. We claim that its image is 
$\mathcal{P}_M$. Indeed, we  notice that the Borel $B_N\subset \mathcal{R}\simeq \mathrm{Aut}(R)$ is contained in 
such image. It remains to show the following:
$$
1+\rho_{ij}\notin \mathcal{P}_M\textrm{ whenever }(i,j)\in\Omega \textrm{ and }j\leq \beta_k< i\textrm{ for some }k,
$$ 
and
$$
1+\rho_{ij}\in \mathcal{P}_M\textrm{ whenever } (i,j)\in\Omega \textrm{ and }\beta_k\leq j< i<\beta_{k+1}\textrm{ for some }k.
$$
To show the first statement, we 
take a pair $(i,j)\in \Omega$ such that $\beta_{l-1}< j\leq \beta_{l}\leq \beta_{k-1}< i\leq \beta_{k}$ 
for some $l$ and $k$. Then $P^j$ is the projective cover of $M(l)$ and $P^i$ is the projective cover of 
$M(k)$ and $P^i=P^j$. So we have $M(l)<M(k)$ and they have the same top. This means that $M(l)=M[r,s]$ 
and $M(k)=M[r,t]$ for some $r\leq s<t$. But then $Q(l)=M[s+1,n]$ and $Q(k)=M[t+1,n]$ and hence the morphism 
$\rho_{ij}$ does not stabilize $\iota(Q)$ since $[Q(k), Q(l)]=0$. 

To show the second statement we 
take a pair $(i,j)\in \Omega$ such that $\beta_k\leq j< i<\beta_{k+1}$. In this case $P^i=P^j$ and $Q^i=Q^j$ 
and hence $\rho_{ij}$ stabilizes $\iota(Q)$. 

To conclude the proof we notice that the map $\Psi$ of Proposition~\ref{Prop:Action} sends 
$1+\rho_{ij}\in \mathcal{P}_M$ to $1+\phi_{ij}\in\mathrm{Aut}(M)$. 
\end{proof}

\begin{rem}
From now on we will freely use Propositions~\ref{Prop:Action} and \ref{Prop:ActionCatenoid}  and we will not distinguish between the action of $B_N$ on $Gr_\mathbf{f}(\xymatrix@1@C=15pt{Q\ar|-{\iota}[r]&R})$ and the action of $\mathrm{Aut}(M)$ on a SQG $Gr_\mathbf{e}(M)$. We will freely say that $B_N$ acts on $Gr_\mathbf{e}(M)$ and that  $\mathrm{Aut}(M)$ acts on  $Gr_\mathbf{f}(\xymatrix@1@C=15pt{Q\ar|-{\iota}[r]&R})$.
\end{rem}

\section{Irreducible components of SQGs}\label{ic}
By definition, a SQG is, in paticular, a closed $B$--stable subvariety of a partial flag variety, and hence its irreducible components are Schubert varieties. In this section we describe these Schubert varieties. 

%Let $\bp=(p_1,\dots,p_r)$ be a Dyck path on $[1,n]$. For an element $p=(i,j)\in \bp$ we denote by $M_p$ 
%the indecomposable representation $M[i,j]$. Let $M=\bigoplus_{l=1}^r M_{p_l}^{m(p_l)}$ be some
%representation supported on $\bp$. We label the summands of $M$ starting from $1$ up to $N=\sum_{p\in\bp} m(p)$
%in such a way that if $i<j$, then the number of $M(p_i)$ is smaller than the number of $M(p_j)$.
%
%\begin{example}
%Let us consider the quiver of type $A_3$ and  the path $\bp=((1,1),(1,2),(2,2),(2,3),(3,3))$. 
%Let $m(p)=1$ for all $p\in\bp$. Then $N=5$ and $M=M[1,1]\oplus M[1,2]\oplus M[2,2]\oplus M[2,3]\oplus M[3,3]$.
%\end{example}

Let $V$ be an $N$-dimensional vector space with a basis $v_1,\dots,v_N$. We consider the algebraic group $SL_N$ acting on $V$, 
the Borel subgroup $B_N$ consisting of upper-triangular matrices, and the maximal torus $T_N$ of diagonal matrices. To the vector space $V$ we attach 
the representation $\bV$ of $A_n$ such that $\bV_i=V$, and such that all the maps are identity maps (in other words, $\bV\simeq P_1^{\oplus N}$). 

Recall the projective resolution  \eqref{Eq:ProjReso} of $M$.
We note that the projective representation $R$ is a sub--representation of $\bV$.
%as follows. 
%For a number $t=1,\dots,N$ let $M[i(t),j(t)]$ be the summand of $M$ labeled by $t$. 
%Then 
%\[
%R=\bigoplus_{t=1}^N M[i(t),n],
%\]  
%where each nonzero space of the module $M[i(t),n]$ is spanned by $v_t$. 
%One has the natural projection $\xymatrix{R\ar[r]^\pi& M}$ and we denote the kernel by $Q$.
%For any dimension vector $\be$ one has an isomorphism 
%\[
%\Gr_\be(M)\simeq \Gr_{\be+\dimv Q} (Q\subset R).
%\] 
%The right hand side is $B_N$ invariant (since $R$ and $Q$ are) and consists of all subrepresentations of 
%$R$ of dimension $\be+\dimv Q$ containing $Q$.  Hence it is isomorphic to the union of Schubert varieties
%inside the partial flag variety $\textrm{Fl}_{\be+\dimv Q}(V)$. 
%We want to describe these Schubert varieties.
Let $\dimv R=(r_1,\dots,r_n)$ and $\dimv Q=(q_1,\dots,q_n)$. Let $W=\mathcal{S}_N$ be the Weyl group of $SL_N$. 
Given a dimension vector $\be$, we set $\mathbf{f}=\mathbf{e}+\dimv Q$, so that $f_a=e_a+q_a$ for all $a=1,2,\cdots, n$. 
\begin{definition}
An element $w\in \mathcal{S}_N$ is called 
\emph{$(M,\be)$--compatible} if 
\[
\{1,\dots,q_a\}\subset w(\{1,\dots,f_a\})\subset \{1,\dots,r_a\} \text{ for all } 1\le a\le n.   
\]
We denote by $W(M,\be)\subset \mathcal{S}_N$ the set of all
$(M,\be)$--compatible elements.
\end{definition}
Recall that, to an element $w\in \mathcal{S}_N$, one attaches the $T_N$-fixed point 
$p(w)\in \textrm{Fl}_{\be+\dimv Q}(V)$ which is the flag  whose $i$--th vector space is $\textrm{span}\,\{v_{w(1)},\cdots, v_{w(i)}\}$.
Two elements $w_1$ and $w_2$ induce the same $T_N$-fixed point (that is, $p(w_1)=p(w_2)$) if and only if
\[
w_1(\{1,\dots,f_a\})=w_2(\{1,\dots,f_a\}),\textrm{ for all }a=1,\dots,n.
\] 
In this case we say that $w_1$ is equivalent to $w_2$.
Let us pick one element of minimal length in each equivalence class and denote this set
by $W^0(M,\be)$.

We define a partial order on $W^0(M,\be)$. For $1\le a\le n$, let 
\[
w(\{1,\dots,f_a\})=\{k^a_1(w),\dots,k^a_{f_a}(w)\} \text{ and } k^a_1(w)<\dots <k^a_{f_a}(w).   
\]
Then we say that $w_1\le w_2$ if $k^a_i(w_1)\le k^a_i(w_2)$ for all $1\le a\le n$, $1\le i\le f_a$.
Note that this order depends on the equivalence classes of $w_1$ and $w_2$, but not on 
the elements themselves. 
\begin{lem}\label{BN}
We have $w_1\le w_2$ if and only if $p(w_1)\in \overline{B_N\,p(w_2)}$. 
\end{lem}
\begin{proof}
The partial order put on $W^0(M,\be)$ is nothing but the Bruhat order and the the claim is hence a well--known fact in the theory of Schubert varieties (see e.g. \cite[Section~10.5]{Fult}). 
\end{proof}

Let $S(M,\be)\subset W^0(M,\be)$ be the set of maximal elements in this poset. 
Lemma \ref{BN} implies that the SQG $\Gr_\be(M)$ is irreducible 
if and only if the set $S(M,\be)$ consists of a single element.   
This case can be described explicitly in the following way. First of all, we perform some reductions: 
%Without loss of generality we can assume that 
%$r_a>f_a$ for all $a$ (i.e. $e_a<\dim M_a$) and $r_a> q_{a+1}$ (i.e. all the maps $M_a\to M_{a+1}$ are non zero).
If $e_a> \dim M_a$ for some vertex $1\leq a\leq n$, then our quiver Grassmannian $\Gr_\be(M)$ is empty. If $e_a = \dim M_a$ 
or $r_a\le q_{a+1}$ (this is equivalent to the fact that the map
$M_a\to M_{a+1}$ is zero), the SQG factors into the product of two quiver Grassmannians (again SQGs) for smaller 
equioriented type $A$ quivers. Similarly, this is the case when $e_a=0$ for some vertex $a$. 
\begin{definition}\label{Def:Simple}
A SQG $Gr_\be(M)$ is called \emph{simple} if $0<e_a<\dim M_a$ and $r_a>q_{a+1}$ for all vertices $a$.
\end{definition}
It is clear from the remark above that every irreducible SQG is a product of simple irreducible SQGs (associated with quivers with less number of vertices). 
\begin{thm}\label{Prop:Irr}
A simple SQG $\Gr_\be(M)$ is irreducible if and only if $r_a-e_a\ge r_{a+1}-e_{a+1}$ for all $a$. 
\end{thm}
\begin{proof}
$\Gr_\be(M)$ is irreducible if and only if $S(M,\be)$ consists of a single element. We note that this happens
if and only if there exists an element $w\in W(M,\be)$ such that
\[
w(\{1,\dots,f_a\})=\{1,\dots,q_a\}\cup \{r_a,r_a-1,\dots,r_a-e_a+1\}.   
\]
This is equivalent to $r_a-e_a\ge r_{a+1}-e_{a+1}$ for all $a$. Indeed,
since $r_{a+1}>f_{a+1}$, the inclusion $w(\{1,\dots,f_a\})\supset \{r_a,r_a-1,\dots,r_a-e_a+1\}$
implies $r_a-e_a+1\le r_{a+1}-e_{a+1}+1$.  
\end{proof}
\begin{cor}
If a simple SQG $\Gr_\be(M)$ is irreducible, the dimension of $\Gr_\be(M)$ is equal to the Euler form 
$\langle \be,\dimv M-\be\rangle_\Gamma$, and there exists a short exact sequence 
$$\xymatrix{0\ar[r]&P\ar[r]&M\ar[r]&I\ar[r]&0}$$ where $P$ is projective, $I$ is injective, and  $\be=\dimv P$.
\end{cor}
\begin{proof}
Assume that $r_a-e_a\ge r_{a+1}-e_{a+1}$ for all $a=1,2,\cdots, n-1$, so that $S(M,\be)$ consists of a single element $w$. We show that the point $p(w)$ of $\mathcal{F}l_\mathbf{f}(R_n)$ corresponds to a projective
sub--representation $P$ of $M$ (under the map $\Psi$ of Proposition~\ref{Prop:IsoQuiveFlagVar}), such that the quotient  $M/P$ is injective. By hypothesis, $e_a\leq e_{a+1}$ for all $a$, and hence there exists a projective representation $P$ of dimension vector $\mathbf{e}$. We have $r_a-e_a\geq r_{a+1}-e_{a+1}>f_{a+1}-e_{a+1}=q_{a+1}$, and hence $e_a+q_{a+1}<r_a$; this implies that $P$ embeds into $M$. To show that the quotient is injective, we note that $r_a-e_a-q_a\geq r_{a+1}-e_{a+1}-q_{a+1}$. 

Finally, the dimension of $\Gr_\be(M)$ is equal to the dimension of the corresponding Schubert variety.
This dimension is equal to 
\[\sum_{a=1}^n (e_a-e_{a-1})(r_a-f_a)=\sum_{a=1}^n (e_a-e_{a-1})(d_a-e_a)=\langle \be,\dimv M-\be\rangle_\Gamma,\]
with the convention that $e_0:=0$.
\end{proof}
\begin{cor}
If a simple SQG $\Gr_\be(M)$ is irreducible, then $e_a\le e_{a+1}$ and $d_a-e_a\ge d_{a+1}-e_{a+1}$.
\end{cor}
The following result provides a representation--theoretic interpretation of Theorem~\ref{Prop:Irr} and shows that the necessary condition for irreducibility of a quiver Grassmannian given in \cite[Theorem~5.1]{KR} is also sufficient for simple SQGs. 
\begin{cor}
A simple SQG $Gr_\be(M)$ is irreducible if and only if  $[M,U]\leq \langle\be,\mathbf{dim}\,U\rangle$ for every non--injective indecomposable representation U.
\end{cor}
\begin{proof}
In view of Theorem~\ref{Prop:Irr}, $Gr_\be(M)$ is irreducible if and only if  
\begin{equation}\label{Cond:Irr}
r_a-e_a\geq r_{a+1}-e_{a+1} \textrm{ for all }a=1,2,\cdots, n-1,
\end{equation} 
where $(r_1,\cdots, r_n)$ is the dimension vector of the projective cover $R$ of $M$. Since $\pi:R\rightarrow M$ is surjective, we have $[M,U]\leq [R,U]$ for all $U$ and it is hence enough to prove that the  conditions \eqref{Cond:Irr} are satisfied if and only if the conditions 
\begin{equation}\label{Cond:Irr2}
[R,U]\leq \langle\be,\mathbf{dim}\, U\rangle \textrm{ for all indecomposable non--injective $U$}
\end{equation} 
are satisfied.  Let $U=M[i,j]$ for some $2\leq i\leq j\leq n$. We have
$[R,U]=r_j-r_{i-1}$
and $\langle\mathbf{e},\mathbf{dim}\,U\rangle=e_j-e_{i-1}$. 
If \eqref{Cond:Irr} holds, then $r_{i-1}-e_{i-1}\geq r_{i}-e_i\geq\cdots\geq r_j-e_j$, and hence 
\eqref{Cond:Irr2} holds. Conversely, if \eqref{Cond:Irr2} holds, then, by choosing $U=S_{a+1}$ ($a=1,2,\cdots, n-1$), 
we get $r_{a+1}-r_a\leq e_{a+1}-e_a$, and hence \eqref{Cond:Irr} holds. 
\end{proof}
Assume that $\Gr_\be(M)$ is irreducible and hence isomorphic to a Schubert variety.
We determine explicitly the (unique) Weyl group element in $S(M,\be)$. 
Denote by $s_i=(i,i+1)$ for $i=1,\dots,N-1$ the simple reflections in the symmetric group $\mathcal{S}_N$. For $1\le a\le n-1$  
let $m_a:=(d_a-e_a)-(d_{a+1}-e_{a+1})$. We define a permutation $\pi_a$ as
\[
(s_{q_a+m_a}\dots s_{q_a+1})\dots (s_{q_a+e_a+m_a-2}\dots s_{q_a+e_a-1}) (s_{q_a+e_a+m_a-1}\dots s_{q_a+e_a}).
\]
For example, for $a=n$ we have
\[
\pi_n=(s_{r_n-e_n}\dots s_{q_n+1})\dots (s_{r_n-2}\dots s_{q_n+e_n-1}) (s_{r_n-1}\dots s_{q_n+e_n})
\]
We note that the number of factors of $\pi_a$ is equal to $e_am_a$ (and equal to $e_n(d_n-e_{n})$
for $\pi_n$).

\begin{prop}\label{Prop:IrrCompIrr}
Whenever a simple SQG $Gr_\mathbf{e}(M)$ is irreducible, the unique element of $S(\be, M)$ is equal to $w= \pi_n\pi_{n-1}\cdots \pi_2\pi_1$.
\end{prop}
\begin{proof}
We need to verify that 
\[
w(\{1,\dots,q_a+e_a\})=\{1,\dots,q_a\}\cup\{r_a-e_a+1,\dots,r_a\},
\] 
and that $w$ is the element of minimal length with this property. The first statement is proved by a direct 
computation (starting from $a=n$ and descending to $a=1$), and the second statement follows from
the equality
\[
\sum_{a=1}^n e_am_a=\langle \be,\bd-\be\rangle_\Gamma.
\]
\end{proof}

\begin{example}[Degenerate flag varieties]\label{Ex:DegFlag} 
Let us consider the catenoid $M$ which is the direct sum of all the indecomposable projective and the indecomposable injectives $\Gamma_n$--representations: 
$$M:=P_1\oplus P_2\oplus\cdots\oplus P_n\oplus I_1\oplus I_2\oplus\cdots \oplus I_n.$$ 
Given the dimension vector $\mathbf{e}=(1,2,\cdots, n)$, the quiver Grassmannian $Gr_\mathbf{e}(M)$ is isomorphic to the $\mathfrak{sl}_{n+1}$--degenerate flag variety \cite{F2, CFR} and it is isomorphic to a Schubert variety $X_{\sigma_n}$ \cite{CL} for a permutation $\sigma_n\in \mathcal{S}_{2n}$. Let us show that $\sigma_n$ coincides with the word $w=w_n$ of Proposition~\ref{Prop:IrrCompIrr} and hence the main result of \cite{CL} becomes a particular case of Theorem~\ref{Prop:Irr} and Proposition~\ref{Prop:IrrCompIrr}. First of all, M is clearly catenoid, $R=P_1^{n+1}\oplus P_2\oplus\cdots\oplus P_n$ and $Q=P_2\oplus\cdots\oplus P_n$ so that $r_a=n+a$, $q_a=a-1$ and $m_a=1$ for all $a=1,2,\cdots, n$. In particular, $r_a-e_a=n$ for all a, and hence $Gr_\mathbf{e}(M)$ is an irreducible SQS (in view of Theorem~\ref{Prop:Irr}). We have $\pi_a=s_as_{a+1}\cdots s_{2a-1}\in\mathcal{S}_{2n}$ which is the permutation given by
$$
\pi_a(k)=\left\{
\begin{array}{ll}
k&\textrm{ if }k\notin\{a,\cdots, 2a\}\\
a&\textrm{ if }k=2a\\
k+1&\textrm{ if }k\in\{a,\cdots, 2a-1.\}
\end{array}
\right.
$$
By definition, 
$$
w=w_n=(s_ns_{n+1}\cdots s_{2n-1})(s_{n-1}s_n\cdots s_{2n-1})\cdots (s_3s_4s_5)(s_2s_3)s_1.
$$ 
We claim that 
\begin{equation}\label{Eq:WnDegFlagVar}
w_n(r)=\left\{
\begin{array}{ll}
k&\textrm{ if }r=2k\\
n+1+k&\textrm{ if }r=2k+1.
\end{array}
\right.
\end{equation}
To prove \eqref{Eq:WnDegFlagVar} one can proceed by induction on $n\geq1$, by noting that $w_n=\pi_nw_{n-1}$ (after identifying $\mathcal{S}_{2n-2}$ with the subgroup of $\mathcal{S}_{2n}$ generated by $s_1,s_2,\cdots, s_{2n-3}$). Formula \eqref{Eq:WnDegFlagVar} shows that $w_n$ is the permutation $\sigma_n$
found in \cite{CL}.  
\end{example}
Now we consider the general case of reducible SQGs.

\begin{thm}\label{Thm:CellDec}
The $T_N$--fixed points of a SQG $\Gr_\be(M)$ are labelled by the elements of $W^0(M,\be)$; the quiver
Grassmannian $\Gr_\be(M)$ admits a cellular decomposition, the cells being the $B_N$--orbits of the $T_N$--fixed point.
The irreducible components of $\Gr_\be(M)$ are parametrized by the elements of $S(M,\be)$.
\end{thm} 
\begin{proof}
Since $\Gr_\be(M)$ is $B_N$--stable, it is equal to a union of several Schubert varieties.
Therefore, a cellular decomposition is provided by the $B_N$--orbits of the $T_N$-fixed points.
Now it suffices to use Lemma \ref{BN}.
\end{proof}

\begin{example}\label{Ex:Complexes}
Let  $M=M[1,1]\oplus\bigoplus_{i=1}^{n-1} M[i,i+1]\oplus M[n,n]$. 
In particular, $\dimv M=(2,\dots,2)$.
Let $\be=(1,\dots,1)$. Then
$N=n+1$, $q_a=a-1$, $r_a=a+1$, and the number of irreducible components of $\Gr_\be(M)$
(that is, the number of elements in $S(\be, M)$) is the $n$-th Fibonacci number. Indeed, assume that
$w\in S(\be,M)$. Then, for all $a$, there are two possibilities: either 
$w(\{1,\dots,a\})=\{1,\dots,a\}$ with no additional restrictions on $w(\{1,\dots,a+1\})$, or
\[
w(\{1,\dots,a\})=\{1,\dots,a-1\}\cup\{a+1\},\; w(1,\dots,a+1)=\{1,\dots,a+1\}. 
\]
To any such $w$ we attach a length $n$ sequence of units and zeroes, where $0$ appears if
$w(\{1,\dots,a\})=\{1,\dots,a\}$, and $0$ appears otherwise. Now the $w\in S(\be,M)$ are labelled
by the length $n$ sequences such that the pair $(1,1)$ is forbidden. The number of such sequences is
exactly the Fibonacci number. We note that the irreducible component attached to such a sequence
is isomorphic to the product of several copies of $\bP^1$, and the number of copies is 
the number of $1$'s in the sequence.  
\end{example}

\section{Poincar\'e polynomials of SQGs}\label{Euler}
Let $\bp$ be a Dyck path on $[1,n]$, and let $M=\bigoplus_{l=1}^r M(p_l)^{m(p_l)}$ be a representation supported on $\bp$. The aim of this 
section is to compute the Poincar\'e polynomial in singular homology of the quiver Grassmannian $\Gr_\be(M)$.

\begin{lem}
Assume that all the indecomposable direct summands of $M$ are isomorphic, that is, there exists a unique $p=(a,b)\in\bp$ such that
$m(p)>0$. Then $\Gr_\be(M)$ is empty unless $e_i=0$ for $i<a$, $i>b$ and $e_a\le\dots\le e_b$. 
If $\Gr_\be(M)$ is non-empty, then the Poincar\'e polynomial of 
$\Gr_\be(M)$ is given by the $q$-multinomial coefficient 
\[
P_{\Gr_\be(M)}(q)=\mn{m(p)}{e_a}{e_b}
\]
\end{lem}
\begin{proof}
If $e_i=0$ for $i<a$, $i>b$ and $e_a\le\dots\le e_b$, then the quiver Grassmannian is isomorphic to the partial flag variety 
$\Fl_{e_a,\dots,e_b}(\bC^{m(p)})$. The cells for the standard cellular decomposition are labeled by collections of subsets 
$K_\bullet=(K_a\subset\dots\subset K_b)$ of a set of cardinality $m(p)$ such that $\#K_i=e_i$;
each cell is the orbit of the Borel subgroup through the $T$-fixed point defined by the sets $K_\bullet$.
\end{proof}
 
Now let us consider the general case. We want to stratify $\Gr_\be(M)$ in such a way that each stratum is
a fibration over a product of partial flag varieties, with fibers being 
affine spaces. Namely,we decompose $\be$ as a sum $\be=\be(1)+\dots + \be(r)$ in such a way that 
\begin{enumerate}
\item \label{s} $\be(l)$ is supported on $p_l$, that is, if $p_l=(i_l,j_l)$, then $\be(l)_a=0$ unless $i_l\le a\le j_l$,
\item \label{l} $\be(l)_a\le\dots\le \be(l)_b\le m(p_l)$.
\end{enumerate}
Then we obtain an embedding
\begin{equation}\label{Pi}
\prod_{l=1}^r \Gr_{\be(l)} (M(p_l)^{\oplus m(p_l)})\subset \Gr_\be(M).
\end{equation}
We denote the image of this embedding by $\Pi(\be(1),\dots,\be(r))$.

Now we consider the subgroups $U(M), L(M)\subset Aut(M)$ (see Lemma \ref{LU}).
We recall that the groups $U(M)$ and $L(M)$ generate $\mathrm{Aut}(M)$, and that
$\Pi(\be(1),\dots,\be(r))$ is $L(M)$-invariant. 

\begin{thm}\label{Thm:Poincare}
We have
\[
\Gr_\be(M)=\bigsqcup U(M) \Pi(\be(1),\dots,\be(r)),
\]
where the disjoint union is taken over all possible sequences $\be(1),\dots,\be(r)$
satisfying \eqref{s} and \eqref{l}. Moreover, the map
\[
U(M) \Pi(\be(1),\dots,\be(r))\to \Pi(\be(1),\dots,\be(r)),\  gx\mapsto x
\]
is an affine space fibration.
\end{thm}  
\begin{proof}
We need to prove that the union $\bigsqcup U(M) \Pi(\be(1),\dots,\be(r))$ covers the whole quiver Grassmannian, and that the map above is indeed a fibration. 

To prove the first claim, consider the embedding 
$\varphi:Gr_\mathbf{e}(M)\rightarrow \mathcal{F}l_{\mathbf{e}+\mathbf{dim}\,Q}(R_n)$ of Proposition~\ref{Prop:IsoQuiveFlagVar}, its image being the  $B_N$--stable subvariety $\Gr_{\be+\dimv Q}(\xymatrix@1@C=8pt{Q\ar^\iota[r]&R})$.
Any point of this image lies in the $B_N$-orbit of some $T_N$-fixed point $\gamma$ of $ \mathcal{F}l_{\mathbf{e}+\mathbf{dim}\,Q}(R_n)$. We note that $\gamma$
belongs to some $\Pi(\be(1),\dots,\be(r))$. Since $\Pi(\be(1),\dots,\be(r))$ is $L(M)$-invariant, and $B_N$ is generated by
$L(M)$ and $U(M)$, the first claim follows. 

Now we consider the map 
$$
U(M)\, \Pi(\be(1),\dots,\be(r))\to \Pi(\be(1),\dots,\be(r)),\qquad gx\mapsto x.
$$
We fix a $T_N$--fixed point $\gamma\in \Pi(\be(1),\dots,\be(r))$ and write
$\gamma=\prod_{l=1}^r\gamma(l)$ according to the factors of \eqref{Pi}.   
There exists an open cell $C(l)$ in each $\Gr_{\be(l)} (M(p_l)^{\oplus m(p_l)})$, containing $\gamma(l)$.
Now it is easy to see that the orbit
$U(M)C(l)$ is canonically isomorphic to the product 
\[
C(l)\times {\rm Hom}_\Gamma(\gamma(l),\bigoplus_{s<l}M(p_s)^{\oplus m(p_s)})),
\]  
where $\gamma(l)$ is considered as a representation of $\Gamma$.
\end{proof}

For $1\le l\le r$ we define $\beta(l)_i=\sum_{s<l} m(p_s)\dim M(p_s)_i$.
\begin{lem}
The dimension $D(\be(1),\dots,\be(r))$ of the fiber of the map $U(M) \Pi(\be(1),\dots,\be(r))\to \Pi(\be(1),\dots,\be(r))$ is equal to
\[
\sum_{l=1}^r \sum_{i=1}^n (e(l)_i-e(l)_{i-1})\beta(l)_i. 
\]
\end{lem}
\begin{proof}
We fix a vector $u$ in the vector space $(M(p_l)^{\oplus m(p_l)})_i$. We note that $\beta(l)_i$ is the dimension of the orbit $U(M) u$.
This implies the claim of the lemma.  
\end{proof}

\begin{cor}
We have the following formula for the Poincar\'e polynomial of the SQG $Gr_\mathbf{e}(M)$:
\[
P_{\Gr_\be(M)}(q)=\sum_{\be(1),\dots,\be(r)} D(\be(1),\dots,\be(r))\prod_{l=1}^r \mn{m(p_l)}{\be(1)}{\be(l)}
\]
\end{cor}

\section*{Acknowledgments}
E.F. was partially supported by the Dynasty Foundation and by the Simons foundation
The article was supported within the framework of a subsidy granted to the HSE by the Government of the Russian Federation
for the implementation of the Global Competitiveness Program
and within the framework of the Academic Fund Program at the National Research University Higher School of Economics (HSE) in
2015--2016 (grant N 15-01-0024).

The work of G.C.I. was financed by the
FIRB project ``Perspectives in Lie Theory'' RBFR12RA9W.

\bibliographystyle{amsplain}

\begin{thebibliography}{99}

\bibitem{BB}
A.~Bia{\l}ynicki-Birula, \emph{Some theorems on actions of algebraic groups},
Ann. of Math. (2) \textbf{98} (1973), 480--497. %\MR{MR0366940 (51 \#3186)}

\bibitem{Bongartz}
K.~Bongartz, \emph{On Degenerations and Extensions of
Finite Dimensional Modules}, Adv. Math. \textbf{121} (1996), 245--287.

\bibitem{CC}
P.~Caldero and F.~Chapoton, \emph{Cluster algebras as {H}all algebras of quiver representations}, 
Comment. Math. Helv. \textbf{81} (2006), no.~3, 595--616.
%\MR{MR2250855 (2008b:16015)}


\bibitem{CR}
P.~Caldero and M.~Reineke, \emph{On the quiver {G}rassmannian in the acyclic
case}, J. Pure Appl. Algebra \textbf{212} (2008), no.~11, 2369--2380.
%\MR{MR2440252 (2009f:14102)}

\bibitem{CL} G.~Cerulli Irelli, M.~Lanini,
{\it Degenerate flag varieties of type A and C are
Schubert varieties}, IMRN (2014). arXiv:1403.2889.
 
\bibitem{CLL}  G.~Cerulli Irelli, M.~Lanini, P.~Littelmann,
{\it Degenerate flag varieties and Schubert varieties: a characteristic free approach},
arXiv:1502.04590.


\bibitem{CFR} G.~Cerulli Irelli, E.~ Feigin, M.~Reineke, \emph{Quiver Grassmannians and degenerate flag varieties}, 
Algebra \& Number Theory \textbf{6} (2012), no. 1, 165--194. arXiv: 1106.2399.

\bibitem{CFR2} G.~Cerulli Irelli, E.~Feigin, M.~Reineke, \emph{Degenerate flag varieties: moment graphs 
and Schr\"oder numbers}, J. Algebraic Combin. \textbf{38} (2013), no. 1. arXiv:1206.4178.

\bibitem{CFR3} G.~Cerulli Irelli, E.~Feigin, M.~Reineke, 
\emph{Desingularization of quiver Grassmannians for Dynkin quivers}, Adv. Math. \textbf{245} (2013), 182--207. arXiv:1209.3960.

\bibitem{Demazure}
M.~Demazure, \emph{D\'esingularisation des vari\'et\'es de {S}chubert  g\'en\'eralis\'ees}, 
Ann. Sci. \'Ecole Norm. Sup. \textbf{7} (1974), 53--88.


\bibitem{F1}
E.~Feigin, \emph{${\mathbb G}_a^M$ degeneration of flag varieties}, 
 Selecta Mathematica, New Series, vol. 18 (2012), no. 3, pp. 513--537.

\bibitem{F2}
E.~Feigin, \emph{Degenerate flag varieties and the median Genocchi numbers}, 
Mathematical Research Letters, no. 18 (6) (2011), pp. 1--16.

\bibitem{FF}
E.Feigin and M.Finkelberg, \emph{Degenerate flag varieties of type A: Frobenius splitting and BWB theorem},
Mathematische Zeitschrift, vol. 275 (2013),  no. 1--2, pp. 55--77.

\bibitem{FFL}
E.~Feigin, G.~Fourier, P.~Littelmann, \emph{PBW filtration and bases for irreducible modules 
in type $A_n$}.  Transf. Groups \textbf{16} (2011), no. 1, 71--89. arXiv:1002.7694.

		
\bibitem{Fult}
W.~Fulton, \emph{Young tableaux},  London Mathematical Society Student Texts, \textbf{35}. Cambridge University Press, (1997).



\bibitem{Hartshorne}
R.~Hartshorne, \emph{Algebraic geometry}, Graduate Texts in Mathematics, vol.~52, Springer, 1977. 

\bibitem{Haupt}
N.~Haupt, \emph{Euler characteristic and geometric properties of quiver Grassmannians}. Ph.D Thesis, University of Bonn. May 2011. 

\bibitem{KR} K.~M\"ollenhoff, M.~Reineke, \emph{Embeddings of representations}. To appear in Algebr. Represent. Theory. arXiv: 1406.5292. 

\bibitem{ReGenExt}
M.~Reineke, \emph{Generic extensions and multiplicative bases of quantum groups
at {$q=0$}}, Represent. Theory \textbf{5} (2001), 147--163 (electronic).
 \MR{1835003 (2002c:17029)}

\bibitem{ReFM}
M.~Reineke,, \emph{Framed quiver moduli, cohomology, and quantum groups}, J.
Algebra \textbf{320} (2008), no.~1, 94--115. \MR{2417980 (2009d:16021)}

\bibitem{SchofieldGeneric}
A.~Schofield, \emph{General representations of quivers}, Proc. London Math.
Soc. (3) \textbf{65} (1992), no.~1, 46--64. \MR{MR1162487 (93d:16014)}


\end{thebibliography}

\end{document}